\documentclass[12pt,reqno,draft]{amsart}
\usepackage{amsmath,amsthm}
\usepackage[T2A]{fontenc}
\usepackage[cp1251]{inputenc}
\usepackage[english,russian]{babel}
\usepackage{amsfonts}
\usepackage{latexsym}

\newcommand*{\cd}{{\,(\cdot)\,}}
\newcommand*{\wx}{\widehat x}

\newcommand*{\la}{\langle}
\newcommand*{\ra}{\rangle}
\newcommand*{\wu}{\widehat u}

\newcommand*{\wou}{\widehat{\overline u}}
\newcommand*{\woa}{\widehat{\overline \alpha}}
\newcommand*{\ov}{\overline}

\newtheorem{theorem}{Theorem}

\newtheorem*{proposition}{Proposition}
\newtheorem*{definition}{Definition}

\begin{document}

\vskip 1cm


\title[Generalized maximum principle]{Generalized maximum principle in optimal control}
\author{E.~R.~Avakov, G.~G.~Magaril-Il'yaev}


\address{Institute of Control Sciences of the Russian Academy of Sciences}
\address{Moscow State University}

\maketitle

For an optimal control problem, the concept of a~strong local infimum is introduce, for
which necessary conditions consisting of some family of ``maximum principles'' are
formulated. If a~function delivers a~strong local minimum in this problem (and therefore,
a~strong local infimum), then this family contains the classical Pontryagin maximum
principle  (see \cite{P},~\cite{IT}). As a~corollary, we derive generalized necessary
conditions for a~strong local minimum for a~problem of the calculus of variations.
Examples  are given to show that the necessary conditions obtained in the present paper
generalize and strengthen classical results.

It is worth noting that some ideas for necessary conditions of this kind are contained in
the book by R.~V.~Gamkrelidze \cite{Gam}, where the Pontryagin maximum principle is
derived as a~corollary to the maximum principle for a~more general problem stated in
terms of generalized controls. We also employ this idea, but from somewhat different
positions and in a~less general setting, when the generalized controls are  sliding
regime control. On the one hand, this constraint is quite sufficient for our purposes,
and on the other hand, it leads to a~simpler set of tools.

\vskip15pt

Let $[t_0,t_1]$ be a line interval, $U$~be a~nonempty subset of~$\mathbb R^r$,
$\varphi\colon \mathbb R\times\mathbb R^n\times \mathbb R^r\to \mathbb R^n$~be a~mapping of
variables $t\in \mathbb R$, $x\in\mathbb R^n$ and $u\in U$, and $f\colon\mathbb
R^n\times\mathbb R^n\to \mathbb R^{m_1}$, $g\colon\mathbb R^n\times\mathbb R^n\to
\mathbb R^{m_2}$ be mappings of variables $\zeta_i\in\mathbb R^n$, $i=1,2$.

Consider the following optimal control problem
\begin{multline}\label{poc}
f_0(x(t_0),x(t_1))\to\min,\quad \dot x =\varphi(t,x,u(t)),\quad u(t)\in U,\\
f(x(t_0),x(t_1))\le0, \quad g(x(t_0),x(t_1))=0,
\end{multline}
where $x\cd\in AC([t_0,t_1],\mathbb R^n)$ (is an absolutely continuous vector function on
$[t_0,t_1]$) and $u\cd\in L_\infty([t_0,t_1],\mathbb R^r)$.

In what follows we assume that the {\it mapping~$\varphi$ is continuous together with its
derivative with respect to~$x$ on~$\mathbb R\times\mathbb R^n\times \mathbb R^r$, and the mappings $f$
and~$g$ are continuously differentiable on  $\mathbb R^n\times\mathbb R^n$}.

A~function $x\cd\in AC([t_0,t_1],\mathbb R^n)$ is called {\it admissible} in problem \eqref{poc} if  $f(x(t_0),x(t_1))\le0$, $g(x(t_0),x(t_1))=0$
and there exists $u\cd\in
L_\infty([t_0,t_1],\mathbb R^r)$ such that $\dot x(t)=\varphi(t,x(t),u(t))$ and $u(t)\in U$ for almost all $t\in[t_0,t_1]$,

\begin{definition} \rm
We say that a function $\wx\cd\in C([t_0,t_1],\mathbb R^n)$ delivers
a~{\it strong local infimum} in problem~\eqref{poc} if  $f(\wx(t_0),\wx(t_1))\le0$,
$g(\wx(t_0),\wx(t_1))=0$, there exist a~neighbourhood~$V$ of the point~$\wx\cd$ and a~sequence
$\{x_N\cd\}$, $N\in\mathbb N$, of admissible functions in~\eqref{poc} such that
$f_0(x(t_0),x(t_1))\ge f_0(\wx(t_0),\wx(t_1))$ for any admissible function $x\cd\in V$ and
$x_N\cd$ converges uniformly to $\wx\cd$ as $N\to\infty$.
\end{definition}

Clearly, if  a~pair $(\wx\cd,\wu\cd)$ delivers a~{\it strong minimum} in problem \eqref{poc}, then $\wx\cd$~is
a~strong local infimum in this problem. On the other hand, if
a~function $\wx\cd$ delivers a~strong local infimum in~\eqref{poc}, $\wx\cd$ is admissible, and
$\wu\cd$ is the corresponding control, then the pair $(\wx\cd,\wu\cd)$ delivers a~strong minimum
in this problem.

Given arbitrary $k\in\mathbb N$ and tuples $\ov u\cd=(u_1\cd,\ldots,u_k\cd)\in
(L_\infty([t_0,t_1]),\mathbb R^r))^k$ and $\ov\alpha\cd=(\alpha_1\cd,\ldots,\alpha_k\cd)\in
(L_\infty([t_0,t_1]))^k$, where $\alpha_i(t)\ge0$, $\alpha_i(t)\ne0$, $i=1,\ldots,k$, and
$\sum_{i=1}^k\alpha_i(t)=1$ for almost all $t\in[t_0,t_1]$,
we associate with the control system specifying the constraints in problem \eqref{poc}
the following {\it extended} (relaxation) control system
\begin{multline}\label{rpoc}
\dot x =\sum_{i=1}^k\alpha_i(t)\varphi(t,x,u_i(t)),\quad u_i(t)\in U,\quad i=1,\ldots,k,\\
f(x(t_0),x(t_1))\le0, \quad g(x(t_0),x(t_1))=0.
\end{multline}

A triple $(x\cd,\ov u\cd,\ov \alpha\cd)$ ($x\cd\in AC([t_0,t_1],\mathbb R^n))$ is called {\it
admissible} for system~\eqref{rpoc} if  it satisfies all constraints in~\eqref{rpoc}.

Let us introduce some notation. We let $\la
\lambda,x\ra=\sum_{i=i}^n\lambda_ix_i$
denote a~linear functional $\lambda=(\lambda_1,\ldots,\lambda_n)\in(\mathbb R^n)^*$
evaluated at a~point
$x=(x_1,\ldots,x_n)^T\in\mathbb R^n$ ($T$~is the transpose). By $(\mathbb R^n)^*_+$ we denote the set
of positive functionals on  $\mathbb R^n$. The adjoint operator to a~linear operator $\Lambda\colon \mathbb R^n\to\mathbb
R^m$ is denoted by~$\Lambda^*$.

Given a~fixed function $\wx\cd$, the partial derivatives of mappings  $f$ and $g$
with respect to $\zeta_1$ and~$\zeta_2$ at a~point $(\wx(t_0),\wx(t_1))$ will be
briefly denoted by $\widehat f_{\zeta_i}$ and $\widehat g_{\zeta_i}$, $i=1,2$, respectively.

\begin{theorem}[generalized maximum principle]\label{T1}
If a function $\wx\cd\in AC([t_0,t_1],\mathbb R^n)$ delivers a~strong local infimum in problem \eqref{poc}, then
for any $k\in\mathbb N$, $\wou\cd=(\wu_1\cd,\ldots,\wu_k\cd)$ and
$\woa\cd=(\widehat\alpha_1\cd,\ldots,\widehat\alpha_k\cd)$ such that the triple
$(\wx\cd,\wou\cd,\woa\cd)$ is admissible for the control system \eqref{rpoc}, there exist
a~nonzero tuple  $(\lambda_0,\lambda_f,\lambda_g)\in \mathbb R_+\times(\mathbb
R^{m_1})^*_+\times(\mathbb R^{m_2})^*$ and a~vector function $p\cd\in AC([t_0,t_1],(\mathbb
R^n)^*)$ such that the following conditions hold:
\begin{itemize}
\item[$1)$] the stationarity condition  with respect to $x\cd$
$$
\dot p(t) =-p(t)\sum_{i=1}^k\widehat\alpha_i(t)\varphi_x(t,\wx(t),\wu_i(t)),
$$
\item[2)] the transversality condition
$$
p(t_0)=\lambda_0{\widehat {f}}_{0\zeta_1}+{\widehat {f}_{\zeta_1}}^*\lambda_f+{\widehat
{g}_{\zeta_1}}^*\lambda_g,\quad p(t_1)=-\lambda_0{\widehat {f}}_{0\zeta_2}-{\widehat
{f}_{\zeta_2}}^*\lambda_f-{\widehat {g}_{\zeta_2}}^*\lambda_g,
$$
\item[3)] the complementary slackness condition
$$
\la \lambda_f, f(\wx(t_0),\wx(t_1))\ra=0,
$$
\item[4)] the maximum condition for almost all $t\in[t_0,t_1]$
\begin{multline*}
\max_{u\in U}\widehat\alpha_i(t)\la p(t),\varphi(t,\wx(t),u)\ra=
\widehat\alpha_i(t)\la p(t),\varphi(t,\wx(t),\wu_i(t))\ra,\\
i=1,\ldots,k,
\end{multline*}
\begin{equation*}
\max_{u\in U}\la p(t),\varphi(t,\wx(t),u)\ra=\la p(t),\dot {\wx}(t)\ra.
\end{equation*}
\end{itemize}

Moreover, if for some $k\in\mathbb N$ and a~triple $(\wx\cd, \ov u\cd,\ov\alpha\cd)$
admissible for the control system~\eqref{rpoc}, conditions {\rm 1)--4)} hold  only when
$\lambda_0\ne0$, then there exists a~sequence of functions $x_N\cd$, $N\in\mathbb N$
admissible in problem~\eqref{poc} such that $x_N\cd\to\wx\cd$ as $N\to\infty$ uniformly
on $[t_0,t_1]$.
\end{theorem}

The first assertion of the theorem constitutes a~family of relations
(parameterized by all possible finite tuples $(\wou\cd,\woa\cd)$ such that
the triple $(\wx\cd,\wou\cd,\woa\cd)$ is admissible for the control system \eqref{rpoc}), of which each
has the form of a~maximum principle. Furthermore, if  $(\wx\cd,\wu\cd)$ is a~strong
minimum in problem \eqref{poc}, then this family contains (with $k=1$, $\wu_1\cd=\wu\cd$ and
$\widehat\alpha_1\cd=1$) the classical Pontryagin maximum principle.

\vskip10pt

As a corollary to Theorem \ref{T1} we obtain generalized conditions for strong local minimum
in the simplest problem of the classical calculus of variations.

Let a~function $L\colon \mathbb R\times\mathbb R^n\times\mathbb R^n\to\mathbb R$
of variables  $t\in\mathbb R$, $x\in\mathbb R^n$ and $\dot x\in\mathbb R^n$ be continuous together
with its partial derivatives with respect to~$x$, $\dot x$, and $x_i\in\mathbb R^n$, $i=0,1$.
Consider the problem
\begin{equation}\label{var1}
\int_{t_0}^{t_1}L(t,x(t),\dot x(t))\,dt\to\min,\quad x(t_0)=x_0,\quad x(t_1)=x_1.
\end{equation}

Given a fixed $\wx\cd$, we write for brevity $\widehat L(t)=L(t,\wx(t),\dot\wx(t))$,
and similarly for the derivatives of~$L$ with respect to $x$ and~$\dot x$.

We let  $\mathcal A^k$ denote the set of tuples
$\ov\alpha\cd=(\alpha_1\cd,\ldots,\alpha_k\cd)$ introduced before the definition of the
control system \eqref{rpoc}.

\begin{theorem}\label{T2}
If a~function $\wx\cd\in AC([t_0,t_1],\mathbb R^n)$ delivers a~strong local minimum in
problem \eqref{var1}, then $\widehat L_{\dot x}\cd\in AC([t_0,t_1],\mathbb R^n)$ and for
any $k\in\mathbb N$, $(\wu_1\cd,\ldots,\wu_k\cd)\in (L_\infty([t_0,t_1],\mathbb R^n))^k$
and $(\widehat\alpha_1\cd,\ldots,\widehat\alpha_k\cd)\in\mathcal A^k$ such that, for
almost all $t\in[t_0,t_1]$
\begin{equation}\label{usl}
\dot\wx(t)=\sum_{i=1}^k\widehat\alpha_i(t)\wu_i(t)\,\,\,\,\text{and}\,\,\,\,\, \widehat
L(t)=\sum_{i=1}^k \widehat\alpha_i(t)L(t,\wx(t),\wu_i(t)),
\end{equation}
the following conditions are satisfied:
\begin{itemize}
\item[$1)$] the generalized Euler equation
$$
-\frac{d}{dt}\widehat L_{\dot x}(t)+\sum_{i=1}^k
\widehat\alpha_i(t)L_x(t,\wx(t),\wu_i(t))=0,
$$
\item[2)] the generalized Weierstrass condition
\begin{multline*}
L(t,\wx(t),u)-L(t,\wx(t),\wu_i(t))-\la L_{\dot
x}(t,\wx(t),\wu_i(t)),u-\wu_i(t)\ra\ge0,\\i=1,\ldots,k,
\end{multline*}
for all $u\in\mathbb R^n$.
\item[3)] If $L$ is twice differentiable with respect to~$\dot x$, then the
generalized Legendre condition holds
$$
\widehat L_{\dot x\dot x}(t,\wx(t),\wu_i(t))\ge0,\quad i=1,\ldots,k.
$$
\end{itemize}
\end{theorem}

It is clear that for  $k=1$, $\wu_1\cd=\dot\wx\cd$ and $\widehat\alpha_1\cd=1$ conditions
\eqref{usl} hold trivially and conditions 1)--3) pass into the classical strong minimum
conditions in problem~\eqref{var1}.

\vskip10pt

The first example illustrates how a~strong local infimum can be found with the help of Theorem~\ref{T1}.

\subsection*{Example 1}

Let $f\colon [0,1]\to\mathbb R$, $g\colon \mathbb R\to\mathbb R$, $m$  be an even number.
Consider the optimal control problem
\begin{multline}\label{exam1}
J(x\cd,u\cd)=\int_0^1((x(t)-f(t))^m+g(u(t)))\,dt\to\min,\quad \dot x=u,\\ |u(t)|
\ge1,\quad x(0)=0,\quad x(1)=f(1).
\end{multline}

We shall assume that the function $f$ is absolutely continuous, $f(0)=0$, and $|\dot f(t)|\le1$,
$|\dot f(t)|\ne1$ for almost all $t\in[0,1]$.
We also assume that~$g$ is continuous on $\mathbb R$, $g(-1)=g(1)$, and $g(u)>g(1)$ for  $|u|>1$.

Our aim is to evaluate the infimum of the functional $J(x\cd,u\cd)$ and find a~sequence
of admissible pairs $(x_N\cd,u_N\cd)$, $N\in\mathbb N$ in problem~\eqref{exam1} on which
the sequence $J(x_N\cd,u_N\cd)$ converges to this infimum.

We transform problem \eqref{exam1} to the equivalent Mayer problem
\begin{multline}\label{exam2}
x_2(1)-x_2(0)\to\min,\quad \dot x_1=u,\quad \dot x_2=(x_1-f(t))^m+g(u),\\
|u(t)|\ge1, \quad x_1(0)=0,\quad x_1(1)=f(1).
\end{multline}

Using Theorem~\ref{T1}, we shall try to find a~function which delivers a~strong local
infimum in this problem. If such a~function $\widehat{\ov x}\cd=(\wx_1\cd,\wx_2\cd)$
is found, then by definition $x_2(1)-x_2(0)\ge\wx_2(1)-\wx_2(0)$ for all admissible
functions $\ov x\cd$ from some neighbourhood of~$\widehat{\ov x}\cd$ and there exists
a~sequence of admissible for~\eqref{exam2} functions  $\ov x_N\cd=(x_{1N}\cd,
x_{2N}\cd)$ that converges uniformly to $\widehat{\ov x}\cd$ as $N\to\infty$. It follows that
$J(x\cd,u\cd)\ge \wx_2(1)-\wx_2(0)$ for all admissible pairs in problem \eqref{exam1} in which $x\cd$ lies
in some neighbourhood of $\wx_1\cd$. Setting
$x_{1N}\cd=x_N\cd$, we find a~sequence of
pairs $(x_N\cd,u_N\cd)$ admissible in problem \eqref{exam1},
where $u_N\cd=\dot x_N\cd$, such that $J(x_N\cd,u_N\cd)\to \wx_2(1)-\wx_2(0)$ as $N\to\infty$.

We apply Theorem~\ref{T1} with $k=2$. By this theorem if  tuples $(\widehat
\alpha_1\cd,\widehat \alpha_2\cd)$ and $(\wu_1\cd,\wu_2\cd)$ are such that
\begin{equation}\label{exam3}
\begin{aligned}
\dot{\wx}_1(t)&=\widehat \alpha_1(t)\wu_1(t)+\widehat\alpha_2(t)\wu_2(t),\\
\dot{\wx}_2(t)&=\widehat \alpha_1(t)g(\wu_1(t))+\widehat\alpha_2(t)g(\wu_2(t))+
(\wx_1(t)-f(t))^m
\end{aligned}
\end{equation}
and $\wx_1(0)=0$, $\wx_1(1)=f(1)$, then there exist a~nonzero set of Lagrange multipliers
$(\lambda_0,\lambda_1,\lambda_2)$, where $\lambda_0\ge0$, and an absolutely continuous function
$p\cd$, such that
\begin{multline}\label{exam4}
\dot p_1=- p_2m(\wx(t)-f(t))^{m-1},\quad \dot p_2=0,\quad p_1(0)=\lambda_1,\quad p_1(1)=-\lambda_2,\\
p_2(0)=p_2(1)=-\lambda_0
\end{multline}
and
\begin{equation}\label{exam5}
\max_{u\in
U}(p_1(t)u+p_2(t)((\wx(t)-f(t))^{m}+g(u)))\!=p_1(t)\dot\wx_1(t)+p_2(t)\dot\wx_2(t)
\end{equation}
for almost all $t\in[t_0,t_1]$.

By examining relations \eqref{exam3}, \eqref{exam4} and \eqref{exam5} one can show that they define uniquely,
up to an additive constant, the function
$(\wx_1\cd,\wx_2\cd)$  (where $\wx_1(t)=f(t)$, $\wx_2(t)=g(1)t+c$ for any
$c\in\mathbb R$ and all $t\in[t_0,t_1]$) and (assuming  $\lambda_0=1$) the Lagrange multipliers
$\lambda=(1,0,0)$ and $p=(0,1)$. Here it suffices to put $\wu_1(t)\equiv 1$,
$\wu_2(t)\equiv-1$, which gives $\widehat\alpha_1(t)=(1+\dot f(t))/2$,
$\widehat\alpha_2(t)=(1-\dot f(t))/2$ for almost all $t\in[t_0,t_1]$.

So, $t\mapsto (f(t),\ g(1)t+c)$ is the only trajectory  suspected for a~strong local infimum in problem~\eqref{exam2} for each
$c\in\mathbb R$. Note that this trajectory is not admissible for this problem. We claim that it delivers a~strong local infimum.

Indeed, $\lambda_0\ne0$, for otherwise relations \eqref{exam4} and \eqref{exam5}
would hold only for $\lambda_1=\lambda_2=0$. Hence, by the second assertion
of the theorem there exists a~sequence of functions
$(x_{1N}\cd, \,x_{2N}\cd)$ which are admissible for problem \eqref{exam2}  and uniformly
converge to $(\wx_1\cd,\, \wx_2\cd)$ as $N\to\infty$.

Additionally, it is clear that  $x_2(1)-x_1(0)\ge g(1)=\wx_2(1)-\wx_2(0)$  for any admissible function $x\cd=(x_{1}\cd, \,x_{2}\cd)$,
and hence by definition $\wx\cd$~delivers the global infimum in problem~\eqref{exam2}.
Hence, the infimum of the functional
$J(x\cd,u\cd)$ is $g(1)$, and by the above there exists a~sequence
of admissible pairs in problem \eqref{exam1} on which this functional converges to~$g(1)$.

Let us construct directly a~sequence of admissible pairs $(x_N\cd\, u_N\cd)$ in problem \eqref{exam1} such that
$J(x_N\cd,u_N\cd)\to g(1)$ as $N\to\infty$. Let $N\in\mathbb
N$. We split the interval $[0,1]$ into  $N$~intervals: $[s/N,\,(s+1)/N]$, $s=0,\ldots,N-1$.
Define $b_N(s)=f(s/N)-(s/N)$ and $c_N(s)=f((s+1)/N)+(s+1)/N$, $s=0,\ldots,N-1$. It is
easily checked that $((c_N(s)-b_N(s))/2)\in [s/N,(s+1)/N]$, $s=0,\ldots,N-1$.

Consider the sequence $x_{N}\cd$ defined by
\begin{equation*}
x_{N}(t)=
\begin{cases} t+b_N(s),\qquad t\in[s/N, \,(c_N(s)-b_N(s))/2],\\[10pt]
-t+c_N(s),\quad \ t\in[(c_N(s)-b_N(s))/2, \,(s+1)/N],
\end{cases}
\end{equation*}
$s=0,\ldots,N-1$.
Each $x_{N}(t)$ is a broken line (with slopes $\pm1$  of the segments and which interpolates $f\cd$ at the points
$s/N$, $s=0,\ldots,N$) and $x_{N}(t)$ uniformly converges to~$f\cd$. The sequence of pairs
$(x_{N}\cd,\,u_N\cd)$, where $u_N\cd=\dot x_{N}\cd$, is admissible in problem \eqref{exam1}, because
$|u_{N}(t)|=1$ for almost all $t\in[t_0,t_1]$, and since  $J(x_{N}\cd, u_{N}\cd)\to g(1)$
as $N\to\infty$, which is clear.

\vskip10pt

The following example shows that even in the classical setting the above assertions are
capable of delivering additional information about the strong minimum compared to with
known necessary conditions.

\subsection*{Example 2}

Let $L\colon \mathbb R\times\mathbb R\to\mathbb R$. Consider the following
variational calculus problem
\begin{equation}\label{z1}
\int_{0}^{1}L(x,\dot x)\,dt\to\min,\quad x(0)=x(1)=0.
\end{equation}
Assume that the function $L$ is continuously differentiable, $L_x(0,0)=L_{\dot x}(0,0)=0$ and
$L(0,\dot x)=0$ for any~$\dot x$ (a~typical situation when $L(x,\dot x)=xf(\dot x)$,
where the function~$f$ is continuously differentiable and $f(0)=0$).

It is an elementary matter to verify that the function $\wx\cd=0$ satisfies the Pontryagin maximum principle
($\dot x=u$, $u\in U=\mathbb R$). The next result is proved using Theorem~\ref{T2}.

\begin{proposition}
If a function $\wx\cd=0$ delivers a~strong local minimum  in problem \eqref{z1}, then the function $\dot x\to L_x(0,\dot x)$ is linear.
\end{proposition}

\begin{proof}
We apply the theorem with $k=2$. It is clear that conditions \eqref{usl} are satisfied for any $u_1<0$, $u_2>0$ and
$\alpha_1=u_2/(u_2-u_1)$, $\alpha_2=-u_1/(u_2-u_1)$. Hence, the Euler equation holds, which in this case reads as
\begin{equation}\label{eqe}
u_2L_x(0,u_1)=u_1L_x(0,u_2).
\end{equation}
Setting here $u_1=-1$, $u_2=1$, we find
\begin{equation}\label{z6}
L_{x}(0,-1)=-L_{x}(0,1).
\end{equation}

Let $u\in\mathbb R$ and $u\ne0$. If $u<0$, then from \eqref{eqe} for $u_1=u$ and $u_2=1$ we find that
$$
L_x(0,u)=L_{x}(0,1)u.
$$
If $u>0$, then again from \eqref{eqe} with $u_2=u$ and $u_1=-1$ and taking into account \eqref{z6}, we have
$$
L_x(0,u)=-L_{x}(0,-1)u=L_{x}(0,1)u.
$$
If $u=0$, then by the condition $L_x(0,0)=0$, and so $L_x(0,u)=L_{x}(0,1)u$ for any $u\in\mathbb R$.
\end{proof}

In fact a more general fact holds. Assume that in problem \eqref{z1}\enskip
$\dot x=u$ and $u(t)\in U$ for almost all $t\in [t_0,t_1]$, where $U$~is an arbitrary set,
but $0\in{\rm int}\,U$. No differentiability of~$L$ with respect to~$u$ is required. If  $\wx\cd=0$
delivers a~strong local infimum in this problem, then using Theorem~\ref{T1} and
arguing as in the proposition, we find that the function $u\mapsto L_x(0,u)$ is linear on
some interval with centre at the origin.

Thus, Theorem \ref{T1} is a~strengthening of the Pontryagin maximum principle.

\smallskip

\textbf{Acknowledgement.} The authors are sincerely grateful to Revaz Valer'yanovich
Gamkrelidze for helpful discussions.

\end{document}